\documentclass{ijmart}

\usepackage{amssymb}    
\usepackage{url}

\theoremstyle{plain}
\newtheorem{theorem}{Theorem}
\newtheorem{problem}{Problem}
\newtheorem{lemma}{Lemma}
\newtheorem{conjecture}{Conjecture}
\theoremstyle{remark}
\newtheorem*{acknowledgment}{Acknowledgment}

\newcommand{\thmlabel}[1]{\label{thm:#1}}   
\newcommand{\lemlabel}[1]{\label{lem:#1}}   
\newcommand{\eqnlabel}[1]{\label{eqn:#1}}   

\newcommand{\thmref}[1]{\ref{thm:#1}}   
\newcommand{\lemref}[1]{\ref{lem:#1}}   
\newcommand{\eqnref}[1]{\eqref{eqn:#1}} 

\newcommand{\inv}{^{-1}}

\newcommand{\Aut}{\mathrm{Aut}}
\newcommand{\Fix}{\mathrm{Fix}}

\begin{document}

\title[Idempotent-Fixing Automorphisms]{Inverse Semigroups with Idempotent-Fixing Automorphisms}

\author{Jo\~{a}o Ara\'{u}jo}
\address[Ara\'{u}jo]{Universidade Aberta\\
1269--001 Lisboa, Portugal \\
and \\
Centro de \'{A}lgebra \\
Universidade de Lisboa \\
1649-003 Lisboa, Portugal}
\email{jaraujo@ptmat.fc.ul.pt}

\author{Michael Kinyon}
\address{Department of Mathematics \\
University of Denver \\ 2360 S Gaylord St \\ Denver, Colorado 80208 USA}
\email{mkinyon@du.edu}

\date{Received on MONTH, YEAR}
\issueinfo{VOL}{NUM}{MONTH}{YEAR}
\doiinfo{10.1007/DOI-NUMBER}
\begin{abstract}
A celebrated result of J. Thompson says that if a finite group $G$ has a fixed-point-free automorphism
of prime order, then $G$ is nilpotent. The main purpose of this note is to extend this result to finite
inverse semigroups. An earlier related result of B. H. Neumann says that a uniquely $2$-divisible group
with  a fixed-point-free automorphism of order $2$ is abelian. We similarly extend this result to
uniquely $2$-divisible inverse semigroups.
\end{abstract}
\maketitle
\tableofcontents

\section{Introduction and main results}
An important result in finite group theory is the following due to J. Thompson \cite{Thompson}.

\begin{theorem}
\thmlabel{Thompson}
Let $G$ be a finite group with a fixed-point-free automorphism of prime order. Then $G$ is nilpotent.
\end{theorem}

The main purpose of this note is to extend this result to finite inverse semigroups.
Standard references for inverse semigroups are (\cite{Howie}, Chap. 5), \cite{Lawson} \cite{Petrich}.
We denote, as usual, the set of idempotents of a semigroup $S$ by $E(S)$,
the automorphism group by $\Aut(S)$,
and the fixed point set of $\alpha \in \Aut(S)$ by $\Fix(\alpha) := \{x\in S\mid x\alpha=x\}$.
Our first main result is the following.

\begin{theorem}
\thmlabel{newmain1}
Let $S$ be  a finite inverse semigroup and let $\alpha\in \Aut(S)$ have prime order and
satisfy $\Fix(\alpha) = E(S)$. Then $S$ is a nilpotent Clifford semigroup.
\end{theorem}

Here, nilpotence of a finite Clifford semigroup is in the sense defined by
Kowol and Mitsch \cite{KowolMitsch}.

An earlier result than Thompson's is the following of B. H. Neumann \cite{Neumann1}.

\begin{theorem}
\thmlabel{neumann}
Let $G$ be a uniquely $2$-divisible group with a fixed-point-free automorphism $\alpha$ of order $2$.
Then $x\alpha = x\inv$ for all $x\in G$ and hence $G$ is abelian.
\end{theorem}

Here uniquely $2$-divisible means that the squaring map $x\mapsto x^2$ is a bijection.
Neumann used this result to prove that a finite group with a fixed-point-free automorphism of
order $2$ must be abelian, for such a group must have odd order and then Theorem \thmref{neumann} applies.
Neumann later outlined a different proof in the finite case in \cite{Neumann2} by observing
that an automorphism $\alpha$ being fixed-point-free is equivalent to the injectivity of the function
$x\mapsto x\inv\cdot x\alpha$. By finiteness, the same function must also be surjective. This together
with $\alpha^2 = 1$ easily implies the desired result. In the same paper, he showed that if one instead
assumes $\alpha^3 = 1$, then $G$ is nilpotent of class $2$.

Theorem \thmref{neumann} is of interest on its own because the hypothesis is independent of
cardinality. Our second main result is to generalize it to inverse semigroups.

\begin{theorem}
\thmlabel{newmain2}
Let $S$ be a uniquely $2$-divisible inverse semigroup and let $\alpha\in \Aut(S)$ satisfy
$\alpha^2 = 1$ and $\Fix(\alpha) = E(S)$. Then $x\alpha = x\inv$ for all $x\in S$ and
hence $S$ is commutative.
\end{theorem}

\section{Proofs of the main results}

In our proofs, we will freely use standard identities in inverse semigroups, such as $(x\inv)\inv = x$ and
the antiautomorphic inverse property $(xy)\inv = y\inv x\inv$ (\cite{Howie}, Proposition 5.12).
In the lemma below, we also use the natural partial order (\cite{Howie}, {\S}5.2).

The critical tool in the  proofs of Theorems \thmref{newmain1} and \thmref{newmain2} is the
following lemma analogous to a key result in \cite{Neumann2}.

\begin{lemma}
\lemlabel{injective}
Let $S$ be an inverse semigroup, let $\alpha\in \Aut(S)$, and
define $\psi : S\to S$ by $x\psi = x\inv\cdot x\alpha$ for all $x\in S$. Then:
\begin{enumerate}
\item If $\Fix(\alpha) = E(S)$, then $\psi$ is injective.
\item If $\psi$ is injective, then $\Fix(\alpha)\subseteq E(S)$.
\end{enumerate}
\end{lemma}
\begin{proof}
For (1), assume $\Fix(\alpha) = E(S)$ and suppose $a\psi = b\psi$ for some $a,b\in S$.
Then $ba\inv\cdot a\alpha = bb\inv\cdot b\alpha$. Applying $\alpha\inv$ to both sides, we get
\begin{equation}
\eqnlabel{above}
(ba\inv)\alpha\inv a = (bb\inv)\alpha\inv b = bb\inv b = b\,.
\end{equation}
Thus $(ba\inv)\alpha\inv aa\inv = ba\inv$. Applying $\alpha$ to both sides of this,
we get (starting with the right side) $(ba\inv)\alpha = ba\inv (aa\inv)\alpha = ba\inv aa\inv = ba\inv$.
Thus $ba\inv \in \Fix(\alpha)$. By hypothesis, $ba\inv \in E(S)$, and so using \eqnref{above}, we get
$ba\inv a = b$, that is, $b\leq a$ (see \cite[Proposition 5.2.1]{Howie}). By the obvious symmetry, we also have $a\leq b$, and thus $a=b$,
which is what we desired to prove.

For (2), suppose $\psi$ is injective and $a\alpha = a$. Then
\[
a\psi = a\inv \cdot a\alpha = a\inv a = a\inv aa\inv a = (a\inv a)\inv \cdot (a\inv a)\alpha = (a\inv a)\psi\,.
\]
By injectivity of $\psi$, $a = a\inv a$, and thus $aa = aa\inv a = a$, as claimed.
\end{proof}

We now prove Theorem \thmref{newmain1}. Recall that if $S$ is an inverse semigroup and $\alpha$ is an automorphism
of $S$, then we have $(a^{-1})\alpha=(a\alpha)^{-1}$. By Lemma \lemref{injective}, the map $\psi$ is injective.
Since $S$ is finite, $\psi$ is also surjective. For $x\in S$, let $y\in S$ satisfy $x = y\psi$. Then
\[
xx\inv = y\psi\cdot (y\psi)\inv =
 y\inv\cdot y\alpha \cdot (y\alpha)\inv\cdot y = y\inv\cdot (yy\inv)\alpha\cdot y
= y\inv \cdot yy\inv\cdot y = y\inv y\,,
\]
and
\[
x\inv x = (y\psi)\inv \cdot y\psi =
 (y\alpha)\inv \cdot y\cdot y\inv\cdot y\alpha = (y\inv\cdot yy\inv\cdot y)\alpha = (y\inv\cdot y)\alpha
= y\inv y\,.
\]
We conclude that $x\inv x = xx\inv$. Therefore $S$ is a completely regular and inverse semigroup, hence it is a Clifford semigroup (see \cite[Theorem 4.2.1]{Howie}).
Now $S = \bigcup G_{\beta}$ is a (strong) semilattice of groups $G_{\beta}$ (see \cite[Theorem 4.2.1]{Howie}). These groups,
which are the $\mathcal{H}$-classes of $S$ (see \cite[Theorem II.1.4]{PandR}), are permuted by $\alpha$ since automorphisms
preserve Green's relations.
But by assumption, $\alpha$ fixes the identity element of each group, and hence $\alpha$ restricts
to an automorphism of each $G_{\beta}$. Now we apply Theorem \thmref{Thompson} to conclude that each group $G_{\beta}$ is
nilpotent. Finally, we appeal to a key feature of the Kowol-Mitsch notion of nilpotence
for finite Clifford semigroups: \emph{if} $S = \bigcup G_{\beta}$ \emph{is a strong semilattice of groups} $G_{\beta}$, \emph{then}
$S$ \emph{is nilpotent if and only if each} $G_{\beta}$ \emph{is nilpotent}
(\cite{KowolMitsch}, Theorem 4.1, p. 442). This completes the proof of Theorem \thmref{newmain1}.

\bigskip

For Theorem \thmref{newmain2}, the squaring map $x\mapsto x^2$ is assumed to be bijective,
and so we denote the unique
square root of an element $x\in S$ by $x^{1/2}$. We have $(x\alpha)^{1/2} = (x^{1/2})\alpha$,
as can be seen immediately from squaring both sides. Similarly, $(x^{1/2})\inv = (x\inv)^{1/2}$,
and we write $x^{-1/2}$ for this common expression.

For the function $x\psi = x\inv\cdot x\alpha$, we note that
\begin{equation}
\eqnlabel{psialpha}
(x\psi)\alpha = (x\alpha)\inv\cdot x\alpha^2=(x\alpha)\inv\cdot x = (x\psi)\inv
\end{equation}
since $\alpha^2 = 1$. Thus we compute
\[
[(x\psi)^{-1/2}]\psi = (x\psi)^{1/2}\cdot ((x\psi)^{-1/2})\alpha =
(x\psi)^{1/2} (x\psi)^{1/2} = x\psi\,,
\]
using \eqnref{psialpha} in the last step. By Lemma \lemref{injective}, $\psi$ is injective,
and so we conclude
\begin{equation}
\eqnlabel{almost}
(x\psi)^{-1/2} = x
\end{equation}
for all $x\in S$. Therefore
\[
x\alpha = [(x\psi)^{-1/2}]\alpha = (x\psi)^{1/2} = [(x\psi)^{-1/2}]\inv = x\inv\,,
\]
using \eqnref{almost} in the first and last equalities, and \eqnref{psialpha} in the second.
Finally, the commutativity of $S$ follows because the inversion mapping $x\mapsto x\inv$ is both
an automorphism (thus $(xy)^{-1}=x^{-1}y^{-1}$) and an antiautomorphism (thus $(xy)^{-1}=y^{-1}x^{-1}$);
now the identities $x^{-1}y^{-1}=y^{-1}x^{-1}$ and $(x^{-1})^{-1}=x$ imply that $xy=yx$.
This completes the proof of Theorem \thmref{newmain2}.

\section{Remarks and Problems}

We do not know if part (2) of Lemma \lemref{injective} extends to a full converse of part (1), that is,
if the injectivity of $\psi$ implies that every idempotent is a fixed point of $\alpha$.

\begin{problem}
Let $S$ be an inverse semigroup and let $\alpha\in \Aut(S)$ satisfy
the property that $x\mapsto x\inv\cdot x\alpha$ is injective. Is it the case that
$\Fix(\alpha) = E(S)$?
\end{problem}

We do have some computational evidence that the answer is affirmative in the following case.

\begin{conjecture}
Let $S$ be an inverse semigroup and let $\alpha\in \Aut(S)$ have finite order and satisfy
the property that $x\mapsto x\inv\cdot x\alpha$ is injective. Then $\Fix(\alpha) = E(S)$.
\end{conjecture}

A property which was useful in our proofs is that an automorphism $\alpha$ of an inverse semigroup
preserves the inversion map, that is, $(x\inv)\alpha = (x\alpha)\inv$. This same property also holds for
completely regular semigroups, and so it is natural to ask if analogs of our results hold in that
setting as well. For instance, we offer the following:

\begin{conjecture}
Let $S$ be a uniquely $2$-divisible completely regular semigroup and let $\alpha\in \Aut(S)$ satisfy
$\alpha^2 = 1$ and $\Fix(\alpha) = E(S)$. Then $x\alpha = x\inv$ for all $x\in S$.
\end{conjecture}

Consideration of the identity mapping on a finite left zero band with at least two elements shows that
one cannot strengthen the conclusion of the conjecture to commutativity.

One might also try regular involuted semigroups, that is, semigroups with a unary operation ${}'$ such that
the identities
\[
(xy)'=y'x'\qquad x''=x \qquad x=xx'x\,.
\]
hold. However, we do not get an immediate generalization of, say, Theorem \thmref{newmain2}.
For instance, let $S$ be the band with the following multiplication table:

\begin{table}[htb]  \centering
\begin{tabular}{r|rrrr}
$\cdot$ & 1 & 2 & 3 & 4\\
\hline
    1 & 1 & 3 & 3 & 1 \\
    2 & 4 & 2 & 2 & 4 \\
    3 & 1 & 3 & 3 & 1 \\
    4 & 4 & 2 & 2 & 4
\end{tabular}
\end{table}

\noindent Let $\alpha$ be the identity mapping on $S$ and let ${}'$ be the unary operation
defined by $1' = 2$, $2' = 1$, $3' = 3$, $4' = 4$. Then $(S,\cdot,{}')$ is a regular involuted
semigroup, but we have neither $x\alpha = x'$ for all $x\in S$ nor that $S$ is commutative.

Despite this, it is certainly reasonable to guess that other classes of regular semigroups might
yield interesting results.

\begin{problem}
Extend Theorem \thmref{neumann} to other classes of regular semigroups.
\end{problem}

Cancellative semigroups form another natural class of semigroups closely related to groups.
Therefore the next problem is very natural.

\begin{problem}
Does the analogue of Theorem \thmref{neumann} hold for cancellative semigroups?
\end{problem}

Regarding nilpotence, we followed the definition of \cite{KowolMitsch} for finite Clifford semigroups.
This definition was motivated by the fact that nilpotence of (finite) groups can
be characterized in various different ways, and the authors of \cite{KowolMitsch} wished to keep these characterizations
in (finite) inverse semigroups (\cite{KowolMitsch}, Main Theorem, p. 448). This is, of course, a rather
strong requirement and suggests why this notion of nilpotence does not extend much beyond Clifford
semigroups.

\begin{problem}
Find appropriate notions of nilpotence for other classes of semigroups, containing the class of all groups,
such that the restriction of the notion to groups is equivalent to the usual one, and, in addition, a
generalization of Theorem \thmref{Thompson} holds for that class of semigroups.
\end{problem}

\bigskip

\begin{acknowledgment}
We are pleased to acknowledge the assistance of the automated theorem prover
\textsc{Prover9} and the finite model builder \textsc{Mace4}, both developed by
W. McCune \cite{McCune}.

The first author was partially supported by FCT and FEDER,
Project POCTI-ISFL-1-143 of Centro de Algebra da Universidade de Lisboa,
and by FCT and PIDDAC through the project PTDC/MAT/69514/2006.
\end{acknowledgment}


\bibliographystyle{ijmart}

\end{document}